\documentclass{amsart}

\usepackage{graphicx}
\newtheorem{thm}{Theorem}[section]
\newtheorem{cor}[thm]{Corollary}

\newtheorem{prop}[thm]{Proposition}
                                                                                                                               \theoremstyle{remark}
\newtheorem{defn}[thm]{Definition}
\theoremstyle{definition}
\newtheorem{exa}[thm]{Example}
\theoremstyle{example}

\numberwithin{equation}{section}

\begin{document}

\title[Induced Representation]{Induced representations of Hilbert $C^*$-modules }%
\author{Gh. Abbaspour tabadkan and S. Farhangi}%
\address{Department of pure mathematics, school of mathematics and computer science, Damghan university, Damghan, Iran.}%
\email{abbaspour@du.ac.ir,  sfarmdst@gmail.com.}%

\date{\today}%
\begin{abstract}
  In this paper, we define the notion of induced representations of a Hilbert $C^{*}$-module and we show that Morita equivalence of two Hilbert modules (in the sense of Moslehian and Joita \cite{JOI}), implies the equivalence of categories of non-degenerate representations of two Hilbert modules.
\end{abstract}
\maketitle
\section{Introduction}
The concept of Morita equivalence was first made by Morita \cite{MOR} in a purely algebraic content. Two unital rings are called Morita equivalent if their categories of left modules are equivalent.\\
This concept has been applied to many different categories in mathematics. And investigate the relationship between an "object", and its "representation theory".\\
In the category of $C^{*}$-algebras, Rieffel \cite{RIE1,RIE2} defined the notions of  induced representations and  (strong) Morita equivalence. The notion of induced representations of $C^*$-algebras, now called Rieffel induction, is to constructing functors between the categories of non-degenerate representations of two $C^*$-algebras. Bursztyn and  Waldmann \cite{BUR}, generalized this notion to $*$-algebras and in 2005 Joita \cite{JOT1} introduced this notion for locally $C^*$-algebras.\\ Two $C^*$-algebras $A$ and $B$ are called Morita equivalent if there exist an $A-B$-imprimitivity. This notion is weaker than isomorphism. There are many valuable papers which study properties of $C^*$-algebras are invariant under the Morita equivalence. (see for examples \cite{ BEE, RIE1, Z})\\
The notion of Morita equivalence in the category of Hilbert $C^*$-modules is defined by Skeide \cite{MSK1} and Joita, Moslehian \cite{JOI}, in two different form. In \cite{JOI}, two Hilbert $A$-module $V$ and Hilbert $B$-module $W$ are called Morita equivalent if the $C^*$-algebras $K(V)$ and $K(W)$ are Morita equivalent as $C^*$-algebras. This notion is weaker than the notion of Morita equivalence defined by Skeide, where he called $V$ and $W$ Morita equivalent when the $C^*$-algebras $K(V)$ and $K(W)$ are isomorphic as $C^*$-algebras.\\
In this paper we show that the weaker notion of Morita equivalence is enough to Hilbert modules have same categories of non-degenerate representations.\\
In section 2, we fix our terminologies and discuss preliminaries about representations of Hilbert modules and Morita equivalence of $C^*$-algebras.\\ In section 3, we introduce the notion of induced representations for Hilbert modules
 and we show that Morita equivalence Hilbert module in the sense of Joita and Moslehian \cite{JOI}, have same categories of non-degenerate representations. \\
  ----------------------------------------------------------------
\section{Preliminary}
A (right) Hilbert $C^*$-module $V$ over a $C^*$-algebra $A$ (or a Hilbert $A$-module ) is by definition
 a linear space that is a right $A$-module, together with an $A$-valued inner
product $\langle . , . \rangle$ on $V \times V$ that is $A$-linear in the second and conjugate linear in the first variable, such that $V$ is a Banach space with the norm define by  $\Vert x\Vert_{A}:= \Vert \langle x , x \rangle_{A} \Vert^{\frac{1}{2}}.$ A Hilbert  $A$-module $V$ is a full Hilbert  $A$-module if the ideal
\begin{center}
 I = $span\lbrace\langle x , y \rangle_{A}  ; x,y \in{X}  \rbrace $
\end{center}
is dense in $A$. The notion of left Hilbert $A$-module is defined in similar way.\\
We denote the $C^*$-algebras of adjointable and compact operators on Hilbert $C^*$-module $V$ by $L(V)$ and $K(V)$, respectively. See \cite{LAN} for more details on Hilbert modules.\\
Now we have a quick review on the notion of Rieffel induction. Let $X$ be a right Hilbert $B$-module  and let $\pi:B\rightarrow B(H)$ be a representation. Then $X\otimes_{alg}H$ is a Hilbert space with inner product
\begin{center}
  $\langle x \otimes h,y \otimes k \rangle$:=$\langle \pi(\langle y,x\rangle_{B})h,k \rangle$
\end{center}
for $x,y \in X$ and $h,k \in H$ [ \cite{RAE}, Proposition 2.64 ].\\
If $A$ acts as adjointable operators on a Hilbert $B$-module $X$, and $\pi$ is a non-degenerate representation of $B$ on $H$. Then $Ind\pi$ defined by $$Ind\pi(a)(x\otimes_{B}h):=(ax)\otimes_{B}h$$ is a representation of $A$ on $X\otimes_{B}H$. If $X$ is non-degenerate as an $A$-module, then $Ind\pi$ is a non-degenerate representation of $A$ [ \cite{RAE}, Proposition 2.66 ].\\ This is a functor from the non-degenerate representations of $B$ to the non-degenerate representations of $A$. Now if we want to get back from representations of $A$ to representations of $B$, we need also an $A$-valued inner product on $X$. This lead us to the following definition.
\begin{defn}
An $A-B$\emph{-imprimitivity bimodule} is an $A-B$-bimodule such that:\\
$(a)$  $X$ is a full left Hilbert $A$-module, and is a full right Hilbert $B$-module;\\
$(b)$ for all $ x,y \in X$, $a \in A$, $b \in B$
\begin{center}
 $\langle ax ,  y \rangle_{B}$=$ \langle x , a^{*}y \rangle_{B} $  and  $_{A}\langle xb ,  y \rangle$=$ _{A}\langle x , yb^{*} \rangle $
\end{center}
$(c)$ for all $ x,y,z \in X$
\begin{center}
 $_{A}\langle x ,  y \rangle z $=$ x \langle y,z \rangle_{B} $.
\end{center}
\end{defn}
If $X$ is an $A-B$-imprimitivity bimodule, let $\widetilde{X}$ be the conjugate vector space, so that there is by definition an additive bijection $b:X\rightarrow \widetilde{X}$ such that $b(\lambda x):= \overline{\lambda}b(x)$. Then $\widetilde{X}$ is a $B-A$-imprimitivity bimodule with\\
$\begin{array}{cc}
  bb(x):=b(xb^{*}) & b(x)a:=b(a^{*}x) \\
  _{B}\langle b(x),b(y)\rangle:=\langle x,y\rangle_{B} & \langle b(x),b(y)\rangle _{A}:=_{A} \langle x,y\rangle
\end{array}$\\
for $x,y \in X$, $a\in A$ and $b \in B$. $\widetilde{X}$ called the \emph{dual module} of $X$.

\begin{exa}
A Hilbert space $H$ is a $K(H)-\textbf{C}$-impirimitivity bimodule with $_{K(H)}\langle h,k \rangle $:=$h \otimes \overline{k}$, where $h \otimes \overline{k}$ denote the rank one operator $g \mapsto \langle g,k\rangle h$.
\end{exa}

\begin{exa}
Every  $C^*$-algebra  $A$ is an $A-A$-imprimitivity bimodule for the bimodule structure given by the multiplication in  $A$, with the inner product  $ _{A}\langle a , b \rangle$=$ab^{*}$ and  $ \langle a , b \rangle_{A}$=$a^{*}b$ .
\end{exa}
Two  $C^*$-algebras  $A$ and $B$ are Morita equivalent if there is an $A-B$-imprimitivity bimodule $X$; we shall say that $X$ implements the Morita equivalence of  $A$ and $B$ .\\
Morita equivalence is weaker than isomorphism. If $\varphi$ is an isomorphism of  $A$ onto $B$, we can construct an imprimitivity bimodule $_{A}X_{B}$ with underlying space $B$ by
\begin{center}
$xb:=xb$, $ax:= \varphi(a)x$, $\langle x,y \rangle_{B}:=x^{*}y$ and $_{A}\langle x,y \rangle:=\varphi^{-1}(xy^{*})$.
\end{center}
Morita equivalence is an equivalence relation on $C^*$-algebras. If $A$ and $B$ are Morita equivalent then the functor mentioned above which comes from tensoring by $X$ has an inverse functor. In fact its inverse is functor comes from tensoring by $\widetilde{X}$, the dual module of $X$ [ \cite{RAE}, Proposition 3.29 ]. So $A$ and $B$ have the same categories of non-degenerate representations.\\
In this paper we will prove that two full Hilbert modules on Morita equivalent $C^*$-algebras have the same categories of  non-degenerate representations. But let us first say some facts about representations of Hilbert modules.\\
Let $V$ and $W$ be Hilbert $C^{*}$-modules over $C^{*}$-algebras $A$ and $B$, respectively, and $ \varphi:A\longrightarrow B$ a morphism of $C^*$-algebras.\\
A map $\Phi :V \longrightarrow W$ is said to be a $\varphi$-morphism of Hilbert $C^*$-modules if\\ $ \langle \Phi(x), \Phi (y) \rangle=\varphi(\langle x,y \rangle)$ is satisfied for all $x,y \in V$.\\
A $\varphi$-morphism $ \Phi :V \longrightarrow B(H,K)$, where $  \varphi :A \longrightarrow B(H) $ is a representation of $A$ is called a representation of $V$. We will say that a representation  $ \Phi :V \longrightarrow B(H,K)$ is a faithful representation of $V$ if $\Phi$ is injective.\\
Throughout the paper, when we say that $\Phi$ is a representation of $V$, we will assume that an associated representation of $A$ is denoted by the same small case letter $ \varphi$.\\
Let  $ \Phi :V \longrightarrow B(H,K)$ be a representation of a Hilbert $A$-module $V$. $\Phi$ is said to be\emph{ non-degenerate} if $\overline{\Phi (V)H}$=$K$ and $\overline{\Phi (V)^{*}K}$=$H$. (Or equivalently, if $\xi_{1} \in H$,$\xi_{2} \in K$ are such that $\Phi (V)\xi_{1}=0$ and $\Phi(V)^{*}\xi_{2}=0$, then $\xi_{1}=0$ and $\xi_{2}=0$ ). If $\Phi$ is non-degenerate, then $\varphi$ is non-degenerate [ \cite{ARA}, Lemma 3.4 ].\\
Let  $ \Phi :V \longrightarrow B(H,K)$ be a representation of a Hilbert $A$-module $V$ and $ K_{1} \prec  H $ , $ K_{2} \prec K $ be closed subspaces.
A pair of subspaces $(K_{1},K_{2})$ is said to be $\Phi$\emph{-invariant} if
\begin{center}
 $\Phi (V)K_{1}\subseteq K_{2}$ and $\Phi (V)^{*}K_{2}\subseteq K_{1}$.
\end{center}
 $ \Phi $ is said to be \emph{irreducible} of $(0,0)$ and $(H,K)$ are the only $\Phi$-invariant pairs.\\
Two representations $\Phi_{i}: V \rightarrow B(H_{i};K_{i})$ of V, $i =1,2$ are said to be (unitarily) equivalent, if there are unitary operators $U_{1}:H_{1}\rightarrow H_{2}$ and $U_{2}:K_{1}\rightarrow K_{2}$; such that  $U_{2}\Phi_{1}(v)=\Phi_{2}(v)U_{1}$
for all $v \in V$. For more details on representations of Hilbert modules see \cite{ARA}.\\
Finally we need the interior tensor product of Hilbert modules, we mention it here briefly. For more details one can refer to the Lance book \cite{LAN}. 
Suppose that $V$ and $W$ are Hilbert $A$-module and  Hilbert $B$-module, respectively, and $\rho :A \longrightarrow L(W)$ is a *-homomorphism, we can regard $W$ as a left $A$-module, the action being given by $(a,y) \longmapsto \rho(a)y$ for all $a \in A$, $y \in W$.\\
We can form the algebraic tensor product of $V$ and $W$ over $A$, $V \otimes_{alg} W$, which is a right $B$-module. The action of $B$ being given by $(x \otimes y)b$ := $x \otimes yb$ for $b \in B$.\\
In fact it is the quotient space of the vector space tensor product $V \otimes_{alg} W$ by the subspace generated by elements of the form
\begin{center}
$xa \otimes y-x \otimes \rho(a)y$, $(x \in V , y \in W , a \in A)$.
\end{center}
$V\otimes_{alg}W$ is an inner product $B$-module under the inner product
\begin{center}
  $\langle x_{1} \otimes y_{1},x_{2} \otimes y_{2} \rangle$=$\langle y_{1},\rho( \langle x_{1},x_{2} \rangle)y_{2}\rangle$
\end{center}
for $x_{1},x_{2} \in V$, $y_{1},y_{2} \in W$.\\
And $V\otimes_{A}W$, which is called the interior tensor product of $V$ and $W$, obtained by completing $V\otimes_{alg}W$ with respect to this inner product.\\

\section{Induced Representation}
In this section we discussed about Morita equivalence of Hilbert $C^{*}$-modules and speak about the notion of induced representation of a Hilbert $C^{*}$-module and then we prove the imprimitivity theorem for induced representations of Hilbert $C^{*}$-modules.
\begin{prop}\label{prop:ind}
Let $V$ and $W$ be two full Hilbert $C^{*}$-modules  over $C^{*}$-algebras $A$ and $B$, respectively. Let $X$ be a $B$-module and $A$ acts as adjointable operators on Hilbert $C^{*}$-module $X$, and $\Phi : W \rightarrow  B(H,K)$ is a non-degenerate representation. Then the formula,
\begin{center}
$ Ind_{X}\Phi(v)(x\otimes g) := v\otimes x \otimes h $
\end{center}
extends to give a representation of $V$ as bounded operator of Hilbert space $X\otimes_{B} H$ to Hilbert space $V\otimes_{A} X\otimes_{B} H$. If $X$ is non-degenerate as an $A$-module, then $Ind_{X}\Phi$ is a non-degenerate representation.
\end{prop}
\begin{proof}
 Since $A$ acts as adjointable operators on the Hilbert $B$-module $X$, so we may construct interior tensor product $V\otimes_{A} X$, which is a $B$-module. Then  $V\otimes_{A} X\otimes_{B} H$ and $X\otimes_{B} H $ are  Hilbert spaces.\\
Let $\Phi :W\rightarrow B(H,K)$ be a non-degenerate representation, so there is a representation $\varphi :B\rightarrow B(H)$ such that $\langle \Phi(x),\Phi(y) \rangle$=$\varphi (\langle x,y \rangle_{B}) $ for all $x,y \in W$. $\varphi$ is non-degenerate so by Rieffel induction we get a non-degenerate representation,
\begin{center}
$Ind_{X}\varphi:A\rightarrow B(X\otimes_{B} H)$.
\end{center}
Now we want to construct a representation, $Ind_{X}\Phi$ of $V$.\\
The mapping $(x,h)\mapsto v\otimes x \otimes h$ is bilinear, thus there is a  linear transformation $\eta_{v}: X\otimes_{alg}H \rightarrow V\otimes_{A} X\otimes_{B} H$ such that $\eta _{v}(x\otimes h)$=$v\otimes x \otimes h$.\\
To see that $\eta_{v}$ is bounded, as in the $C^*$-algebraic case, we may suppose that $\varphi$ is cyclic, with cyclic vector $h$. Then for any $x_{i} \in X , b_{i} \in B$ we have
 \begin{align*}
  \|\eta_{v}(\sum_{i=1}^{n}x_{i} \otimes \varphi(b_{i})h) \|^{2} &= \sum_{i} \sum_{j}\langle v\otimes x_{i} \otimes \varphi(b_{i})h,v\otimes x_{j} \otimes \varphi(b_{j})h \rangle \\
  & = \sum_{i} \sum_{j}\langle  x_{i} \otimes \varphi(b_{i})h,Ind_{X}\varphi (_{A}\langle v,v \rangle) x_{j} \otimes \varphi(b_{j})h \rangle \\
  & = \sum_{i} \sum_{j}\langle  x_{i} \otimes \varphi(b_{i})h,_{A}\langle v,v \rangle x_{j} \otimes \varphi(b_{j})h \rangle \\
  & = \sum_{i} \sum_{j}\langle \varphi(b_{i})h,\varphi (\langle _{A}\langle v,v \rangle x_{i}, x_{j}\rangle_{B}) \varphi(b_{j})h \rangle \\
  & = \sum_{i} \sum_{j}\langle h,\varphi(b_{i}^{*})\varphi (\langle _{A}\langle v,v \rangle x_{i}, x_{j}\rangle_{B})\varphi(b_{j})h \rangle \\
  & = \sum_{i} \sum_{j}\langle h,\varphi (\langle _{A}\langle v,v \rangle x_{i}b_{i}, x_{j}b_{j}\rangle_{B})h \rangle \\   & = \sum_{i} \sum_{j}\langle h,\varphi(\langle _{A}\langle v,v \rangle^{\frac{1}{2}} x_{i}b_{i},_{A}\langle v,v \rangle^{\frac{1}{2}} x_{j}b_{j}\rangle_{B})h \rangle \\
  & = \langle h,\varphi(\langle _{A}\langle v,v \rangle^\frac{1}{2}\sum_{i} x_{i}b_{i},_{A}\langle v,v \rangle^\frac{1}{2} \sum_{j} x_{j}b_{j}\rangle_{B})h \rangle \\
  & \leq \|_{A}\langle v,v \rangle^\frac{1}{2}\|^{2}\langle h,\varphi(\langle \sum_{i} x_{i}b_{i},  \sum_{j}x_{j}b_{j}\rangle_{B})h \rangle \\
  & = \|v\|_{A}^{2}\langle \sum_{i} x_{i}b_{i}\otimes h,\sum_{j} x_{j}b_{j}\otimes h \rangle \\
  & = \|v\|_{A}^{2}\| \sum_{i} x_{i} \otimes \varphi(b_{i})h\|^{2} .
\end{align*}
So $\eta_{v}$ is bounded and $\|\eta_{v}\|^{2}\leq \|v\|_{A}^{2}$.\\
Hence $\eta_{v}$ extends to an operator $Ind_{X}\Phi(v)$ on $X\otimes_{B} H$ and we have\\
\begin{align*}
\langle x \otimes h,Ind_{X}\Phi^{*}(v) Ind_{X}\Phi(v^{'}) x^{'}\otimes h^{'} \rangle & = \langle Ind_{X}\Phi(v)(x\otimes h),Ind_{X}\Phi(v^{'})(x^{'}\otimes h^{'}) \rangle \\
& = \langle v\otimes x\otimes h,v^{'}\otimes x^{'}\otimes h^{'} \rangle \\
& = \langle x\otimes h, Ind_{X}\varphi(_{A}\langle v,v^{'}\rangle)x^{'}\otimes h^{'} \rangle .
\end{align*}
Thus  $$\langle Ind_{X}\Phi(v), Ind_{X}\Phi(v^{'}) \rangle=Ind_{X}\Phi^{*}(v) Ind_{X}\Phi(v^{'})=Ind_{X}\varphi(_{A}\langle v,v^{'}\rangle).$$So $Ind_{X}\Phi:V \rightarrow B(X\otimes_{B} H,V\otimes_{A} X\otimes_{B} H)$ is an $Ind_{X}\varphi$-morphism and hence a representation of $V$.\\
Now we show that $Ind_{X}\Phi$ is non-degenerate. For this we must to show that\\ $\overline{Ind_{X}\Phi(V)X\otimes_{B} H}$=$V\otimes_{A} X\otimes_{B} H$ and $\overline{Ind_{X}\Phi(V)^{*}(V\otimes X\otimes h)}$=$X\otimes H$.\\
By definition of $Ind_{X}\Phi$ it is easy to see that $\overline{Ind_{X}\Phi(V)X\otimes_{B} H}$=$V\otimes_{A} X\otimes_{B} H$.\\
By hypotheses,  $A$ acts as adjointable operators on Hilbert $C^{*}$-module $X$ and this action is non-degenerate, that is, $\overline{AX}$=$X$ and $V$ is full, $\overline{\langle V, V\rangle}$=$A$, so $\overline{\langle V,V\rangle X}$=$X$.\\
For all $x\otimes h\in X\otimes_{alg}H$ we have:
$$\|x\otimes h\|^{2}=|\langle h,(\varphi(\langle x,x\rangle_{B}))h\rangle |\leq \|\langle x,x\rangle_{B}\|\|h\|^{2}=\|x\|_{B}^{2}\|h\|^{2},$$
so if $\sum _{A}\langle v_{i},v^{'}_{i}\rangle x_{i}$ approximates $x$, then $\sum _{A}\langle v_{i},v^{'}_{i}\rangle x_{i}\otimes h$ approximates $x\otimes h$.\\
But, $\sum _{A}\langle v_{i},v^{'}_{i}\rangle x_{i}\otimes h$=$\sum Ind_{X}\varphi(_{A}\langle v_{i},v^{'}_{i}\rangle)x_{i}\otimes h$=$\sum Ind_{X}\Phi(v_{i})^{*}Ind_{X}\Phi(v^{'}_{i})(x_{i}\otimes h)$.\\
So every elementary tensor $x\otimes h$ in $X\otimes_{alg}H$ can be approximated by a sum of the form $\sum Ind_{X}\Phi(v_{i})^{*}Ind_{X}\Phi(v^{'}_{i})(x_{i})\otimes h$=$\sum Ind_{X}\Phi(v_{i})^{*}(v^{'}_{i}\otimes x_{i}\otimes h)$.\\
Thus $Ind_{X}\Phi$ is a non-degenerate representation.
\end{proof}
\begin{defn}
We call the representation $ Ind_{X}\Phi$ constructed  above, the Rieffel-induced representation from $W$ to $V$ via $X$.   
\end{defn}
\begin{prop}
Let $\Phi_{1} :W\rightarrow B(H_{1},K_{1})$ and $\Phi_{2} :W\rightarrow B(H_{2},K_{2})$ be two non-degenerate
representations. If $\Phi_{1}$ and $\Phi_{2}$ are unitarily equivalent, then $Ind_{X}\Phi_{1}$ and $Ind_{X}\Phi_{2}$  are unitarily equivalent.
\end{prop}
\begin{proof} Suppose $U_{1}:H_{1}\rightarrow H_{2}$ and $U_{2}:K_{1}\rightarrow K_{2}$ be unitary operators
such that $U_{2}\Phi_{1}(w)=\Phi_{2}(w)U_{1}$. Then $id_{X}\otimes U_{1}:X\otimes_{alg} H_{1}\rightarrow X\otimes_{alg} H_{2}$ given by $x\otimes h\mapsto x\otimes U_{1}(h)$ and $id_{V}\otimes id_{X}\otimes U_{2}:V\otimes_{A} X\otimes_{alg} H_{1}\rightarrow V\otimes_{A} X\otimes_{alg} H_{2}$ given by $v\otimes x\otimes h\mapsto v\otimes x\otimes U_{1}(h)$ may be extended  to unitary operators $V_{1}$ from $X\otimes_{B} H_{1} $ onto $X\otimes_{B} H_{2} $ and $V_{2}$ from $V\otimes_{A}X\otimes_{B} H_{1}$ onto $V\otimes_{A}X\otimes_{B} H_{1}$ and moreover, $V_{2}Ind_{X}\Phi_{1}(v)=Ind_{X}\Phi_{2}(v)V_{1}$. So $Ind_{X}\Phi_{1}$ and $Ind_{X}\Phi_{2}$ are unitarily equivalent.
\end{proof}

\begin{cor}
Suppose $\Phi:W \rightarrow B(H,K)$ and $\oplus_{s}\Phi_{s}:W \rightarrow B(\oplus_{s}H_{s},K)$ are unitary equivalent, then $Ind_{X}\Phi:V\rightarrow B(X\otimes_{B}H,V\otimes_{A}X\otimes_{B}H)$ is unitary equivalent to $\oplus_{s}Ind_{X}\Phi_{s}:V\rightarrow B(X\otimes_{B}\oplus_{s}H_{s},V\otimes_{A}X\otimes_{B}\oplus_{s}H_{s})$.\\
\end{cor}
\begin{defn}[ \cite{JOI}, Definition 2.1 ]
Two Hilbert $C^*$-modules $V$ and $W$ , respectively, over $C^*$-algebras
$A$ and $B$ are called Morita equivalent, if the $C^*$-algebras $K(V)$ and $K(W)$ are Morita equivalent as $C^*$-algebras. 
\end{defn}
It is well known that for Hilbert $C^*$-module $V$, $K(V)$ is Morita equivalent to $\overline{\langle V, V\rangle}$, so if $V$ and $W$ are full, then they are Morita equivalent if and only if their underlying $C^*$-algebras are Morita equivalent [ \cite{JOI}, Proposition 2.8 ].\\
The following theorem show that there is a bijection  between non-degenerate representations of two Morita equivalent full Hilbert $C^*$-modules. The fullness property is not crucial, if necessary we can replace underlying $C^*$-algebra by a suitable ones,  thus two Morita equivalent Hilbert modules in the above sense have same categories of non-degenerate representations. 
\begin{thm} \label{thm:cor}
Suppose that $X$ is an $A-B$-imprimitivity bimodule, and $\Phi$ and $\Psi$ are non-degenerate representation of $W$ and $V$, respectively. Then $Ind_{\widetilde{X}}(Ind_{X}\Phi)$ is naturally unitary equivalent to $\Phi$, and $Ind_{X}(Ind_{\widetilde{X}}\Psi)$ is naturally unitary equivalent to $\Psi$.
\end{thm}
\begin{proof} If $\Phi:W\rightarrow B(H,K)$ is a non-degenerate representation by proposition \ref{prop:ind}, $Ind_{X}\Phi:V\rightarrow B(X\otimes_{B}H,V\otimes_{A}X\otimes_{B}H)$ is a non-degenerate representation of $V$.\\ Again usage of proposition \ref{prop:ind} to $Ind_{X}\Phi$ instead of $\Phi$, give us the following  non-degenerate representation of $W$, $$Ind_{\widetilde{X}}(Ind_{X}\Phi):W\rightarrow B(\widetilde{X}\otimes_{A}X\otimes_{B}H,W\otimes_{B}\widetilde{X}\otimes_{A}X\otimes_{B}H).$$ Now we want to show that $\Phi$ is unitary equivalent to $Ind_{\widetilde{X}}(Ind_{X}\Phi)$. By the proof of Theorem 3.29 \cite{RAE}; $U_{1}:\widetilde{X}\otimes_{A}X\otimes_{B}H\rightarrow H$ defined by $b(x)\otimes y\otimes h\mapsto \varphi(\langle x,y \rangle_{B})h$ is a unitary operator.\\ We define $$U_{2}:W\otimes_{B}\widetilde{X}\otimes_{A}X\otimes_{B}H\rightarrow K$$ given by $$w\otimes b(x)\otimes y\otimes h\mapsto\Phi(w)\varphi(\langle x,y \rangle_{B})h.$$ $U_{2}$ is a unitary operator, and we have, 
\begin{align*}
U_{2}Ind_{\widetilde{X}}(Ind_{X}\Phi(w))(b(x)\otimes y\otimes h)&=U_{2}(w\otimes \varphi(\langle x,y\rangle_{B})h)\\&=\Phi(w)\varphi(\langle x,y \rangle_{B})h \\&=  \Phi(w)U_{1}(b(x)\otimes y\otimes h). 
\end{align*}
So $$U_{2}Ind_{\widetilde{X}}(Ind_{X}\Phi(w))=\Phi(w)U_{1}.$$ Hence $\Phi$ and $Ind_{\widetilde{X}}(Ind_{X}\Phi)$ are unitary equivalent.\\
 For the equivalence of $\Psi$ and $Ind_{X}(Ind_{\widetilde{X}}\Psi)$, apply the first part to $_{B}\widetilde{X}_{A}$ instead of $_{A}X_{B}$.
\end{proof}

The Financial Support of the Research Council of Damghan University 
 with the Grant Number 92093 is Acknowledged.

\bibliographystyle{amsplain}

\end{document}